\documentclass[a4paper,12pt]{article}
\usepackage{amsmath,amssymb,amsthm,geometry,mathrsfs}
\usepackage{hyperref}
\usepackage[utf8]{inputenc}
\usepackage{listings}
\usepackage{graphicx}
\usepackage{xcolor}

\newtheorem{theorem}{Theorem}[section]

\title{A note on the theoretical approach to Grassmannians and Plücker coordinates for additive skew-symmetric pairwise comparisons matrices}
\author{Waldemar W. Koczkodaj, Witold Pedrycz, Alexander Pigazzini, \\ Laura P. Pigazzini and Richard Pincak}

\date{January 2025}

\begin{document}

\maketitle

\begin{abstract} 
\noindent Symmetry and antisymmetry are fundamental concepts in many strict sciences. Pairwise comparisons (PC) matrices are fundamental tools for representing pairwise relations in decision making. In this theoretical study, we present a novel framework that embeds additive skew-symmetric PC matrices into the Grassmannian manifold \( G(2, n) \). This framework leverages Plücker coordinates to provide a rigorous geometric interpretation of their structure.
\\
Our key result demonstrates that the algebraic consistency condition \( a_{ij} + a_{jk} - a_{ki} = 0 \) is equivalent to the geometric consistency of 2-planes in \( G(2, n) \), satisfying the Plücker relations. This connection reveals that the algebraic properties of PC matrices can be naturally understood through their geometric representation.  
Additionally, we extend this framework by interpreting PC matrices as differential 2-forms, providing a new perspective on their consistency as a closedness condition. Our framework of linear algebra, differential geometry, and algebraic geometry, placing PC matrices in a broader mathematical context.  
Rather than proposing a practical alternative to existing methods, our study aims to offer a theoretical foundation for future research by exploring new insights into higher-dimensional geometry and the geometric modeling of pairwise comparisons.
\end{abstract}

\noindent \textbf{keyword:}Pairwise comparison, inconsistency, skew-symmetric; Grassmannian; Plücker coordinates; differential 2-forms.

\section{Introduction} 
Symmetry and antisymmetry are fundamental concepts in many strict sciences. Pairwise comparisons (PC) matrices have emerged as fundamental tools in decision-making and beyond, offering a structured means to represent and analyze pairwise relationships. By translating subjective judgments into matrix form, PC matrices facilitate the extraction of priorities and the study of relational structures. \\
A fundamental aspect of PC matrices is the consistency condition. Historically, consistency was first introduced in a multiplicative form, expressed as \( m_{ij} \cdot m_{jk} = m_{ik} \). Later, through a logarithmic transformation, it was reinterpreted in an additive form, \( a_{ij} + a_{jk} - a_{ik} = 0 \). Conversely, the additive form can be transformed back into the multiplicative form by applying the exponential function, \( m_{ij} = e^{a_{ij}} \).
\\
n fact, let us consider the following additive skew-symmetric pairwise comparisons matrix:
\[
A = \begin{bmatrix}
0 & 1 & 0 \\
-1 & 0 & -1 \\
0 & 1 & 0 \\
\end{bmatrix}.
\]
\noindent This matrix is consistent because, for any triad \((i, j, k)\), the condition \(a_{ij} + a_{jk} - a_{ik} = 0\) holds.
For example, the triad \((1, 0, -1)\) satisfies \(1 + (-1) \)= 0.
\\
The transposed matrix preserves the consistency.
\\
The equivalent multiplicative form of this matrix is:

\[
M = \begin{bmatrix}
1 & e & 1 \\
1/e & 1 & 1/e \\
1 & e & 1 \\
\end{bmatrix}.
\]

\noindent Here, the multiplicative consistency condition \(m_{ij} \cdot m_{jk} = m_{ik}\) is satisfied in the same way.
\\
\\
Traditional approaches to PC matrices rely on optimization and eigenvalue-based techniques, to evaluate their properties and derive meaningful outcomes. While effective, these methods often obscure the deeper geometric and structural relationships inherent to PC matrices. To address this, our study develops a novel theoretical framework that bridges the algebraic and geometric perspectives of PC matrices.  
\\
This paper focuses on additive skew-symmetric PC matrices, a class of matrices characterized by their simplicity and direct interpretability. By embedding these matrices into the Grassmannian manifold \( G(2, n) \), which represents the space of all 2-dimensional subspaces of \( \mathbb{R}^n \), we uncover a natural geometric representation of the consistency condition \( a_{ij} + a_{jk} - a_{ki} = 0 \). Specifically, we demonstrate that this algebraic condition is equivalent to the geometric consistency of the associated subspaces in \( G(2, n) \), satisfying the Plücker relations.  
\\
The key contributions of this work are as follows: 
\begin{itemize}

\item Geometric embedding: 
We provide a rigorous mathematical framework for representing additive skew-symmetric PC matrices as elements of the Grassmannian \( G(2, n) \).  

\item Linking consistency and geometry: 
We establish that the algebraic consistency of PC matrices corresponds directly to the geometric properties of 2-dimensional subspaces, as encoded by Plücker coordinates. 

\item Interdisciplinary insights: 
By connecting PC matrices with Grassmannians, our work bridges linear algebra, differential geometry, and algebraic geometry, offering new perspectives for the analysis of pairwise comparisonss.  
\end{itemize}

\noindent The intricate relationships between Grassmannians and algebraic or geometric structures have been the subject of extensive research. Pankov et al., in \cite{Pankov2006}, explored the geometry of polar Grassmann spaces. In \cite{Liu}, the transformations that preserve the intersection relations are analyzed, while in \cite{Pankov2014}, isometric embeddings of Grassmann graphs are studied. Prazmowska et al., in \cite{Prazmowska2010}, examined the Grassmann spaces of regular subspaces. Additional studies have examined the Euclidean geometry of orthogonality in subspaces (\cite{Prazmowska2008}), and provided an axiomatic description of Strambach planes (\cite{Prazmowski1989}). Pankov, in \cite{Pankov2018}, investigated the geometric version of Wigner's theorem in the context of Hilbert Grassmannians. Zynel, in \cite{Zynel2014}, explored complements of Grassmann substructures in projective Grassmannians. Govc et al., in \cite{Govc}, investigated the minimal triangulations of Grassmannians, providing a combinatorial perspective essential for understanding their topological complexity.
\\
In parallel, in the context of pairwise comparison (PC) matrices, advances have focused on new definitions of consistency (\cite{K1993}), statistically accurate improvements in error rates (\cite{K1996}), effective methods for inconsistency reduction (\cite{KS2016}), the development of inconsistency indicator maps on groups (\cite{KSW2016}), and ordinal consistency indicators for PC matrices (\cite{Kuo2021}). Recent work has also highlighted the societal relevance of PC matrices, such as their use in modeling electronic health record breaches (\cite{SIR2019}). Building on these diverse contributions, this paper bridges Grassmannian geometry and PC matrices by embedding additive skew-symmetric PC matrices into the Grassmannian \( G(2, n) \), offering new perspectives on their geometric consistency and algebraic properties.
\\
Recent advancements in the study of Grassmannians and related algebraic structures have opened new avenues for mathematical modeling and theoretical insights. In \cite{Balinsky2024}, the authors analyzed quantum deformations of associative Sato Grassmannian algebras and identified their relation to bi-Hamiltonian flows on smooth functional submanifolds, providing a foundation for understanding Frobenius manifolds in topological field theory. Similarly, \cite{Yamawaki2023} explored the role of rho mesons as dynamical gauge bosons within the Grassmannian model of hidden local symmetry, emphasizing large \(N\)-dynamics as a dual to QCD. In \cite{Calixto2022}, the authors investigated the Hilbert space structure of low-energy sectors in \(U(N)\) quantum Hall ferromagnets, demonstrating a connection to Grassmannian phase spaces. Further, the work in \cite{Fioresi2021}, extended this paradigm to quantum supertwistors, embedding super-Minkowski spaces into Grassmannians, while Planat et al., in \cite{Planat2020}, explored exotic space-times linked to Grassmannians for quantum computation and measurement frameworks. Lastly, Jung et al., in \cite{Jung2019}, characterized spherical submanifolds, including the Clifford torus and great sphere, through Grassmannian-derived spherical Gauss maps. These interdisciplinary contributions highlight the versatility of Grassmannian frameworks. In the same way, our work aims at the growth of this research area that connects Grassmannians to diverse mathematical and physical frameworks. By embedding additive skew-symmetric pairwise comparison matrices into the Grassmannian \( G(2, n) \), we provide a novel perspective that bridges linear algebra, differential geometry, and algebraic geometry.
\\
This work does not aim to introduce a practical approach alternative to traditional methods for analyzing pairwise comparison matrices, but aims to place them in a new mathematical context, building a bridge between PC matrix theory and algebraic, differential and multilinear geometry. In this way, we explore new theoretical perspectives that broaden their scope of application and study.

\section{Mathematical foundations} 
 
A pairwise comparisons (PC) matrix \( A = [a_{ij}] \) is an \( n \times n \) matrix where \( a_{ij} \) represents the relative importance of alternative \( i \) over \( j \). 
\\
\\
It has the following properties:

\begin{itemize}

\item Multiplicative form: \( a_{ij} > 0 \), \( a_{ij} \cdot a_{ji} = 1 \).  

\item Additive form: \( a_{ij} + a_{ji} = 0 \), leading to a skew-symmetric matrix:  
\[
A^T = -A.  
\]  
\end{itemize}

A matrix is consistent if:  
\[
a_{ij} + a_{jk} - a_{ki} = 0 \quad \text{(additive form)}.  
\]  
\noindent Consistency in assessments is illustrated by Figure \ref{fig1}

\begin{figure}
    \centering
    \includegraphics[width=0.7\linewidth]{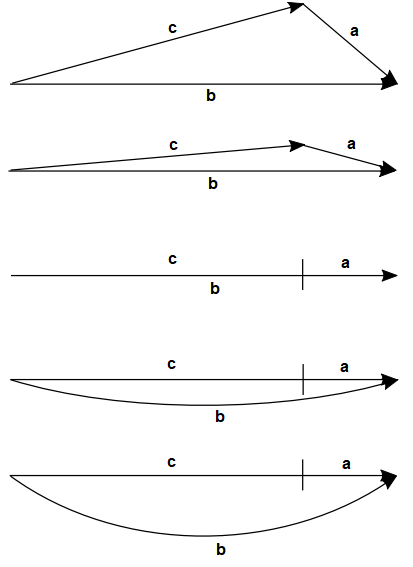}
    \caption{Inconsistency in additive pairwise comparisons}
    \label{fig1}
\end{figure}

\newpage

\subsection{Skew-symmetric matrices}

A skew-symmetric matrix \( A \in \mathbb{R}^{n \times n} \) satisfies \( A^T = -A \). Key properties include: 
\begin{itemize}

\item All diagonal elements are zero (\( a_{ii} = 0 \)).  
\item The eigenvalues of \( A \) are purely imaginary or zero.  
\item \( A \) defines a bilinear form \( B(\mathbf{u}, \mathbf{v}) = \mathbf{u}^T A \mathbf{v} \), which is antisymmetric.  
\end{itemize}

\subsection{Grassmannians and geometric representation of PC matrices}  
The Grassmannian \( G(k, n) \) is the manifold representing all \( k \)-dimensional subspaces of the Euclidean space \( \mathbb{R}^n \). In this work, we focus on \( G(2, n) \), which describes the set of all 2-dimensional planes passing through the origin in \( \mathbb{R}^n \). The use of \( G(2, n) \) is motivated by the geometric nature of additive Pairwise comparisons PC matrices, which can naturally be interpreted in terms of 2-dimensional subspaces.

\subsection*{Fundamental concepts}
Each 2-dimensional subspace in \( \mathbb{R}^n \) is "spanned" by a pair of linearly independent vectors. However, this representation using spanning vectors is not unique, as the choice of basis vectors for a subspace can vary. To overcome this ambiguity, we use the wedge product (or exterior product), which provides a unique representation of the subspace.
\\
Given two vectors \( v_i, v_j \in \mathbb{R}^n \), their wedge product \( v_i \wedge v_j \) is an element of an associated exterior space, capturing all the information about the plane defined by \( v_i \) and \( v_j \). Specifically, the wedge product can be interpreted as an oriented measure of the 2-dimensional subspace spanned by \( v_i \) and \( v_j \).
\\
The Grassmannian \( G(2, n) \) can therefore be viewed as a geometric manifold that parametrizes these 2-dimensional subspaces, making it an ideal tool to represent pairwise relationships. For a detailed description of \( G(2, n) \), with related geometric tools, we recommend \cite{Tu}.

\subsection*{Plücker coordinates}

We now introduce the concept of Plücker coordinates. This tool will be needed to embed the subspaces of the Grassmannian into a higher-dimensional Euclidean space.
\\
Plücker coordinates are an algebraic representation that describes \( k \)-dimensional subspaces in \( \mathbb{R}^n \) space. For the case \( k = 2 \), they describe two-dimensional planes spaced by two linearly independent vectors \( v_i, v_j \in \mathbb{R}^n \).

If \( v_i = (v_{i,1}, v_{i,2}, \ldots, v_{i,n}) \) and \( v_j = (v_{j,1}, v_{j,2 }, \ldots, v_{j,n}) \), the Plücker coordinates \( p_{kl} \) are given by the formula:
\[
p_{kl} = v_{i,k} v_{j,l} - v_{i,l} v_{j,k}, \quad \text{for } 1 \leq k < l \leq n.
\]

\noindent Each pair \( (k, l) \) represents a contribution to the two-dimensional plane generated by the vectors \( v_i \) and \( v_j \). Then, \( p_{kl} \) measures the oriented area of the parallelogram projected onto the plane determined by the axes \( e_k \) and \( e_l \).
\\
These coordinates capture the essence of the subspace generated by \( v_i \) and \( v_j \) and are scale invariant of the vectors, that is:
\[
p_{kl} \text{ remains unchanged if we multiply } v_i \text{ or } v_j \text{ by a constant factor.}
\]
\noindent Plücker coordinates and their quadratic relations are discussed in depth in \cite{Harris}.

\subsubsection*{Relation to the wedge product}
The wedge product \( v_i \wedge v_j \) is a linear combination of the bases \( e_k \wedge e_l \) of the outer space \( \bigwedge^2 \mathbb{R}^n \), with coefficients equal to \( p_{kl} \). Mathematically:
\[
v_i \wedge v_j = \sum_{1 \leq k < l \leq n} p_{kl} \, e_k \wedge e_l.
\]

\noindent The Plücker coordinates \( p_{kl} \) are therefore the coefficients that represent the wedge product in the canonical basis of \( \bigwedge^2 \mathbb{R}^n \).
\\
\\
To understand better, we propose the following simple example.
\\
Let's consider \( v_i = (1, 0, 0) \) and \( v_j = (0, 1, 0) \) in \( \mathbb{R}^3 \). Let's calculate the Plücker coordinates \( \{p_{12}, p_{13}, p_{23}\} \):
\begin{itemize}
\item  \( p_{12} = v_{i,1} v_{j,2} - v_{i,2} v_{j,1} = (1)(1) - (0)(0) = 1 \),

\item \( p_{13} = v_{i,1} v_{j,3} - v_{i,3} v_{j,1} = (1)(0) - (0)(0) = 0 \),

\item \( p_{23} = v_{i,2} v_{j,3} - v_{i,3} v_{j,2} = (0)(0) - (0)(1) = 0 \).
\end{itemize}

\noindent The Plücker coordinates are then:
\[
p_{12} = 1, \quad p_{13} = 0, \quad p_{23} = 0.
\]

\noindent These values indicate that the plane generated by \( v_i \) and \( v_j \) is the one in the \( x \) and \( y \) axes, with no projection onto the plane containing the \( z \) axis.
\\
\\
Plücker coordinates satisfy the following properties:
\begin{itemize}
    
\item Homogeneity: Plücker coordinates are proportional to the scale of the vectors. For example, multiplying \( v_i \) or \( v_j \) by a constant \( c \) results in the \( p_{kl} \) coordinates being multiplied by \( c^2 \).

\item Compatibility condition: The Plücker coordinates satisfy a series of quadratic relations (called Plücker Relations), which guarantee that they truly represent a two-dimensional subspace.

\end{itemize}

\subsection*{Plücker embedding}
To work computationally with the Grassmannian \( G(2, n) \), it is helpful to embed its subspaces into a higher-dimensional Euclidean space, \( \mathbb{R}^{\binom{n}{2}} \). This embedding, known as the Plücker embedding, uses Plücker coordinates to represent subspaces while preserving their geometric relationships.
\\
\\
For \( G(2, n) \):
\\
\\
- Each 2-dimensional subspace is associated with a vector \( p_{ij} \in \mathbb{R}^{\binom{n}{2}} \), whose components are determined by the linear combinations of the spanning vectors.
\\
- These coordinates allow the measurement of distances, angles, and geometric deviations between subspaces in a rigorous and computationally efficient manner.

\subsection*{Motivation for choosing \( G(2, n) \)}
The adoption of the Grassmannian \( G(2, n) \) for representing additive PCMs is motivated by several factors:

\begin{itemize}

\item Natural Relationship with Skew-Symmetric Matrices: The components of an additive PC matrix can be interpreted as geometric descriptors of the planes generated by pairs of alternatives. These planes naturally fit into the structure of \( G(2, n) \).

\item Geometric Intuition: Inconsistency in a PC matrix manifests as a deviation of the subspaces represented by \( p_{ij} \) from a common hyperplane in \( \mathbb{R}^{\binom{n}{2}} \). This deviation is easily quantifiable using the geometric framework of \( G(2, n) \).

\item Computational Efficiency: The Plücker embedding transforms complex geometric problems into linear or quadratic operations in Euclidean space, simplifying the computations of projections and corrections.

\item Generality: Although this work focuses on \( G(2, n) \), the approach can be extended to higher-dimensional Grassmannians to address more complex or multidimensional problems.

\end{itemize}

\subsection*{Representing additive PC matrices via \( G(2, n) \)}
For a PC matrix \( A = [a_{ij}] \), each element \( a_{ij} \) can be interpreted as a measure of the pairwise comparisons between alternatives \( i \) and \( j \). Within the proposed framework:
\\
\\
- Each alternative \( i \) is associated with a vector \( v_i \in \mathbb{R}^n \).
\\
- The comparisons between \( i \) and \( j \) generates a 2-dimensional subspace represented by \( w_{ij} = v_i \wedge v_j \in G(2, n) \).
\\
- The Plücker coordinates \( p_{ij} \) describe this subspace in \( \mathbb{R}^{\binom{n}{2}} \), enabling the geometric analysis of matrix inconsistency.
\\
\\
In summary, the Grassmannian \( G(2, n) \) provides a rigorous and intuitive framework for representing and correcting additive PCMs by leveraging the geometry of manifolds to quantify and reduce pairwise comparisons inconsistencies.

\section{Consistency of a skew-symmetric additive PC matrix and Plücker coordinates in the Grassmannian \( G(2, n) \)}

\begin{theorem}
Let \( A = [a_{ij}] \) be a skew-symmetric additive PC matrix of order \( n \), defined by \( a_{ij} = s_i - s_j \), where \( s_i = \|v_i\|^2 \in \mathbb{R} \) represents a scalar measure associated with option \( i \). There exists a set of vectors \( v_i = b_i e_i \), where \( \{e_i\}_{i=1}^n \) is an orthonormal basis of \( \mathbb{R}^n \), and \( b_i \neq 0 \) for all \( i \), such that:
\\
\\
1. Each element \( a_{ij} \) is given by the difference \( a_{ij} = \|v_i\|^2 - \|v_j\|^2 \).
\\
2. The algebraic consistency of the matrix, \( a_{ij} + a_{jk} - a_{ki} = 0 \), for every \( i, j, k \), is equivalent to the geometric consistency of the \( 2 \)-planes generated by every pair of vectors \( (v_i, v_j) \), represented in the Grassmannian \( G(2, n) \) by points whose Plücker coordinates satisfy the Plücker relations.
\end{theorem}

\begin{proof}
    
By definition, the elements of the matrix \( A \) are given by:
\[
a_{ij} = s_i - s_j = \|v_i\|^2 - \|v_j\|^2,
\]
where \( s_i = \|v_i\|^2 \) and \( \{v_i\}_{i=1}^n \subset \mathbb{R}^n \) are linearly independent vectors.
\\
The matrix is skew-symmetric:
\[
a_{ij} + a_{ji} = (\|v_i\|^2 - \|v_j\|^2) + (\|v_j\|^2 - \|v_i\|^2) = 0.
\]

\noindent The algebraic consistency condition is:
\[
a_{ij} + a_{jk} - a_{ki} = 0, \quad \forall i, j, k.
\]
\noindent Substituting \( a_{ij} = \|v_i\|^2 - \|v_j\|^2 \), we get:
\[
\|v_i\|^2 - \|v_j\|^2 + \|v_j\|^2 - \|v_k\|^2 - \|v_k\|^2 + \|v_i\|^2 = 0.
\]
The terms are cancelled as per the definition of consistency.
\\
Each pair of vectors \( (v_i, v_j) \) generates a \( 2 \)-plane in \( \mathbb{R}^n \), represented by the wedge product:
\[
v_i \wedge v_j = \sum_{1 \leq k < l \leq n} (v_{i,k}v_{j,l} - v_{i,l}v_{j,k}) e_k \wedge e_l,
\]
\noindent where \( v_{i,k} \) is the \( k \)-th component of \( v_i \).
\\
The wedge product \( v_i \wedge v_j \) is uniquely associated with the 2-plane generated by \( v_i \) and \( v_j \), which is an element of the Grassmannian \( G(2, n) \).
\\
Now consider the relation on wedge products derived from the algebraic condition:
\[
(v_i \wedge v_j) + (v_j \wedge v_k) - (v_k \wedge v_i).
\]
\noindent We explicitly calculate the coefficients \( C_{kl} \) for each term, starting from the base \( \{e_k \wedge e_l\} \):
\[
C_{kl} = (v_{i,k}v_{j,l} - v_{i,l}v_{j,k}) + (v_{j,k}v_{k,l} - v_{j,l}v_{k,k}) - (v_{k,k}v_{i,l} - v_{k,l}v_{i,k}),
\]
\noindent collecting we can rewrite it as:
\[
C_{kl} = v_{i,k}(v_{j,l} + v_{k,l}) + v_{j,k}(v_{k,l} - v_{i,l}) - v_{k,k}(v_{i,l} + v_{j,l}).
\]
\noindent Since \( a_{ij} + a_{jk} - a_{ki} = 0 \), the bilinear terms cancel, thanks to the cyclic combination, so:
\[
C_{kl} = 0.
\]
\noindent Since \( C_{kl} = 0 \) has been proved explicitly, it follows that the wedge products cancel:
\[
(v_i \wedge v_j) + (v_j \wedge v_k) - (v_k \wedge v_i) = 0.
\]
\noindent This implies that algebraic consistency is equivalent to geometric consistency.
\\
The Plücker coordinates \( p_{kl} \) associated with \( v_i \wedge v_j \) are:
\[
p_{kl} = v_{i,k}v_{j,l} - v_{i,l}v_{j,k}.
\]
\noindent From geometric consistency, the \( p_{kl} \) satisfy the Plücker relations:
\[
p_{kl}p_{mn} - p_{km}p_{ln} + p_{kn}p_{lm} = 0, \quad \forall k, l, m, n.
\]
\noindent The geometric consistency of the wedge products automatically implies the validity of the Plücker relations, since these derive from the linearity of the wedge product and from the condition:
\[
(v_i \wedge v_j) + (v_j \wedge v_k) - (v_k \wedge v_i) = 0.
\]
To explicitly verify the Pl\"ucker relations, we substitute the Pl\"ucker coordinates into the quadratic condition. For indices $k, l, m, n$:
\[
    p_{kl}p_{mn} - p_{km}p_{ln} + p_{kn}p_{lm} = 0,
\]
\noindent where \(p_{kl} = v_{i,k}v_{j,l} - v_{i,l}v_{j,k}\). Substituting directly:
\[
p_{kl}p_{mn} = (v_{i,k}v_{j,l} - v_{i,l}v_{j,k})(v_{m,k}v_{n,l} - v_{m,l}v_{n,k}),
\]

\[
p_{km}p_{ln} = (v_{i,k}v_{j,m} - v_{i,m}v_{j,k})(v_{l,k}v_{n,l} - v_{l,m}v_{n,k}),
\]

\[
p_{kn}p_{lm} = (v_{i,k}v_{j,n} - v_{i,n}v_{j,k})(v_{l,m}v_{n,l} - v_{l,n}v_{n,m}).
\]

\noindent The expansion of these terms reveals that each of the products and cross-terms align due to the cancellation imposed by \((v_i \wedge v_j) + (v_j \wedge v_k) - (v_k \wedge v_i) = 0\). 
\\
This direct computation confirms:

\[
    p_{kl}p_{mn} - p_{km}p_{ln} + p_{kn}p_{lm} = 0,
\]
\noindent proving that the Pl\"ucker relations are satisfied.
\end{proof}

\noindent We have shown that the algebraic consistency of a skew-symmetric and additive PC matrix is equivalent to the geometric consistency of the wedge products \( v_i \wedge v_j \), represented in the Grassmannian \( G(2, n) \). Furthermore, the geometric consistency guarantees the validity of the Plücker relations, highlighting the connection between algebraic structure and geometric representation.

\section{Relation between geometric and algebraic inconsistency}

As seen above, a pairwise comparisons (PC) matrix \(A = [a_{ij}]\) is algebraically consistent if, for every triple \((i, j, k)\), the relation holds:
\[
\Delta_{ijk} = a_{ij} + a_{jk} - a_{ik} = 0.
\]
\noindent When this relation is not satisfied, we define the total algebraic inconsistency as:
\[
I_{\text{alg}} = \sum_{i,j,k} (\Delta_{ijk})^2.
\]

\noindent From \textit{Theorem 3.1}, we have seen that geometric consistency is satisfied if, for each triple \((i, j, k)\), the wedges produced satisfy the relation:
\[
(v_i \wedge v_j) + (v_j \wedge v_k) - (v_k \wedge v_i) = 0,
\]
\noindent where \(v_i, v_j, v_k \in \mathbb{R}^n\) are the vectors associated with the elements of the matrix \(A\). When this relation is not satisfied, we define the geometric deviation as:
\[
\Delta_{\text{geom}} = (v_i \wedge v_j) + (v_j \wedge v_k) - (v_k \wedge v_i).
\]
\noindent The total geometric inconsistency is then:
\[
I_{\text{geom}} = \| \Delta_{\text{geom}} \|^2,
\]
\noindent where \(\|\cdot\|\) is the quadratic norm in the 2-form space \(\wedge^2 \mathbb{R}^n\).
\\
Let us develop \(\Delta_{\text{geom}}\) explicitly in terms of the coordinates of the vectors \(v_i\). We know that:
\[
v_i \wedge v_j = \sum_{1 \leq k < l \leq n} (v_{i,k}v_{j,l} - v_{i,l}v_{j,k}) e_k \wedge e_l.
\]
\noindent Therefore:
\[
\Delta_{\text{geom}} = \sum_{1 \leq k < l \leq n} \Delta_{kl} e_k \wedge e_l,
\]
\noindent where the coefficients \(\Delta_{kl}\) are given by:
\[
\Delta_{kl} = (v_{i,k}v_{j,l} - v_{i,l}v_{j,k}) + (v_{j,k}v_{k,l} - v_{j,l}v_{k,k}) - (v_{k,k}v_{i,l} - v_{k,l}v_{i,k}).
\]

\noindent Each term \(\Delta_{kl}\) depends linearly on the algebraic deviations \(\Delta_{ijk}\). Let us introduce the explicit relation:
\[
\Delta_{kl} = \sum_{i,j,k} C_{kl}^{ijk} \cdot \Delta_{ijk},
\]
\noindent where \(C_{kl}^{ijk}\) are linear combination coefficients that depend on the structure of the vectors \(v_i, v_j, v_k\) and on the relation between their indices. The indices (\(i, j, k\)) indicate the components of the algebraic deviation \(\Delta_{ijk}\), while the indices (\(k, l\)) indicate the components of the geometric deviation \(\Delta_{kl}\).
\\
The general form for \(C_{kl}^{ijk}\) is:
\[
C_{kl}^{ijk} = (\delta_k^i \delta_l^j - \delta_l^i \delta_k^j) 
+ (\delta_k^j \delta_l^k - \delta_l^j \delta_k^k) 
- (\delta_k^k \delta_l^i - \delta_l^k \delta_k^i),
\]
\noindent where:

\[
C_{kl}^{ijk} =
\begin{cases} 
1 & \text{if } \big( k = i \land l = j \big) \lor \big( k = j \land l = k \big) \lor \big( l = k \land i = k \big), \\ 
-1 & \text{if } \big( l = i \land k = j \big) \lor \big( l = j \land k = k \big) \lor \big( k = k \land l = i \big), \\ 
0 & \text{otherwise}.
\end{cases}
\]

\noindent Then, the geometric inconsistency can be rewritten as:
\[
I_{\text{geom}} = \sum_{k < l} (\Delta_{kl})^2 = \sum_{k < l} \left( \sum_{i,j,k} C_{kl}^{ijk} \cdot \Delta_{ijk} \right)^2.
\]

\noindent Similarly, we can rewrite the algebraic inconsistency as:
\[
I_{\text{alg}} = \sum_{i,j,k} (\Delta_{ijk})^2.
\]
\noindent We define \(\mathbf{\Delta}\) as the vector of algebraic deviations \(\Delta_{ijk}\). Then, we can represent \(I_{\text{geom}}\) in matrix form:
\[
I_{\text{geom}} = \mathbf{\Delta}^\top M \mathbf{\Delta},
\]
\noindent where \(M\) is a positive semidefinite matrix constructed from the coefficients \(C_{kl}^{ijk}\).
\\
Since \(M\) is positive semidefinite, minimizing \(I_{\text{geom}} = \mathbf{\Delta}^\top M \mathbf{\Delta}\) necessarily implies minimizing the norm of \(\mathbf{\Delta}\), that is:
\[
\|\mathbf{\Delta}\|^2 = I_{\text{alg}}.
\]
\noindent This result guarantees that by minimizing \(I_{\text{geom}}\), one also minimizes \(I_{\text{alg}}\). Furthermore, the proportionality between \(I_{\text{geom}}\) and \(I_{\text{alg}}\) is guaranteed by the linear structure of the coefficients \(C_{kl}^{ijk}\).

\subsection{Analysis of the construction of \( M \) and implications for minimizing the inconsistency}

The matrix \( M \) is constructed only in the case in which the matrix PC \( A = [a_{ij}] \) does not satisfy the condition of algebraic consistency, that is when there exist triples \( (i, j, k) \) for which:
\[
\Delta_{ijk} = a_{ij} + a_{jk} - a_{ki} \neq 0.
\]
\noindent In this case, the proposed geometric method aims to reduce the algebraic inconsistency by minimizing the geometric inconsistency defined as:
\[
I_{geom} = \|\Delta_{geom}\|^2,
\]
\noindent where \(\Delta_{geom}\) is the geometric deviation associated with the external products of vectors \( v_i = b_i e_i \).
\\
When \( A \) is consistent (\( \Delta_{ijk} = 0 \, \forall i, j, k \)), the matrix \( M \) is not needed, since \( I_{geom} = 0 \) automatically. However, in the inconsistent case \( A \), it is necessary to compute \( M \) to formalize the minimization.
\\
\\
Let \( A = [a_{ij}] \) be an inconsistent PC matrix, with vectors \( v_i = b_i e_i \), where \( \{e_i\} \) is an orthonormal basis of \( \mathbb{R}^n \) and \( b_i \neq 0 \). The geometric deviation \(\Delta_{geom}\) is defined as:
\[
\Delta_{geom} = (v_i \wedge v_j) + (v_j \wedge v_k) - (v_k \wedge v_i),
\]
\noindent and can be represented as:
\[
\Delta_{geom} = \sum_{1 \leq k < l \leq n} \Delta_{kl} \, e_k \wedge e_l,
\]
\noindent where the coefficients \(\Delta_{kl}\) are expressed as:
\[
\Delta_{kl} = \sum_{i,j,k} C_{kl}^{ijk} \Delta_{ijk}.
\]

\noindent The coefficients \( C_{kl}^{ijk} \) represent the linear combinations between algebraic deviations \(\Delta_{ijk}\) and geometric deviations \(\Delta_{kl}\). The quadratic norm of geometric inconsistency is therefore:
\[
I_{geom} = \sum_{k<l} (\Delta_{kl})^2 = \sum_{k<l} \left( \sum_{i,j,k} C_{kl}^{ijk} \Delta_{ijk} \right)^2.
\]

\noindent We can rewrite as:
\[
I_{geom} = \sum_{k<l} \sum_{i,j,k} \sum_{p,q,r} C_{kl}^{ijk} C_{kl}^{pqr} \Delta_{ijk} \Delta_{pqr}.
\]

\noindent We define the matrix \( M \) as:
\[
M_{ijk, pqr} = \sum_{k<l} C_{kl}^{ijk} C_{kl}^{pqr},
\]
\noindent which describes the interaction between the algebraic deviations \(\Delta_{ijk}\). In matrix form:
\[
I_{geom} = \Delta^\top M \Delta,
\]
\noindent where \(\Delta\) is the vector containing all the \(\Delta_{ijk}\) terms.
\\
\\
Inconsistency reduction by minimizing \( I_{geom} \) is guaranteed only if the matrix \( M \) is non-degenerate, i.e. if its rank is maximum with respect to the size of \(\Delta\). This implies that \( \text{rank}(M) = \dim(\Delta) \).
\begin{itemize}

\item If \( M \) is non-degenerate:
\\
\\
1. The minimization \( \min I_{geom} = \Delta^\top M \Delta \) yields a unique solution for \(\Delta\), ensuring the global reduction of the geometric inconsistency and, consequently, an improvement in the algebraic inconsistency \( I_{alg} \).
\\
\\
2. When \( M \) is non-degenerate, there is direct proportionality between \( I_{geom} \) and \( I_{alg} \), as shown in Section 4:
\[
I_{alg} = \|\Delta\|^2, \quad I_{geom} = \Delta^\top M \Delta.
\]

\item If \( M \) is degenerate:
\\
\\
1. The minimization is not unique, since there is a non-trivial kernel \( \ker(M) \). The components of \(\Delta\) belonging to the kernel of \( M \) do not contribute to \( I_{geom} \), making the inconsistency reduction partial or ineffective.
\\
\\
2. Not all algebraic deviations \(\Delta_{ijk}\) are represented geometrically, weakening the connection between \( I_{geom} \) and \( I_{alg} \).

\end{itemize}

\noindent The construction of the matrix \( M \) is necessary only if \( A \) is inconsistent. However, the inconsistency reduction method via minimization is fully effective only if \( M \) is non-degenerate. Otherwise:
\\
- Minimization can be ambiguous, with non-unique solutions.
\\
- The reduction of the algebraic inconsistency \( I_{alg} \) can only be partial.
\\
\\
Summarizing, the practical validity of the method depends on the rank of \( M \), which in turn is influenced by the distribution of the \( b_i \) (the preferences). To ensure the effectiveness of the minimization method, it is necessary that the \( b_i \) are sufficiently distinct and that the size of the space \( \mathbb{R}^n \) is adequate with respect to the number of triples \( (i, j, k) \).

\subsection{Management of M matrix degeneration cases}

In cases where \( M \) is degenerate, minimizing the geometric inconsistency may encounter ambiguity or ineffectiveness. To deal with such cases, several strategies can be adopted. A common approach is to regularize \( M \), adding a small diagonal term \( \lambda I \) to ensure non-degeneracy, although this may slightly affect the geometric interpretation. Alternatively, one can introduce additional constraints that reduce the solution space, such as imposing normalization or partial consistency on some triples. Finally, a preliminary diagnosis of the rank of \( M \) can help identify the directions of degeneracy and limit ambiguity by exploiting techniques such as eigenvalue analysis or principal component decomposition.

\section{Discussion and new perspectives for additive PC matrices} 
The \textit{theorem 3.1} states that the algebraic consistency of an additive PC matrix (\( a_{ij} + a_{jk} - a_{ki} = 0 \)) is equivalent to the geometric consistency of the two-dimensional subspaces generated by pairs of vectors \( v_i \) and \( v_j \), represented in the Grassmannian \( G(2, n) \). Furthermore:
\\
\\
- Each pair of vectors generates an exterior product \( v_i \wedge v_j \), which defines an oriented subspace.
\\
- The Plücker coordinates of the generated subspaces satisfy the Plücker relations, ensuring global geometric consistency.
\\
\\
This result naturally connects the following branches of mathematics:
\begin{itemize}

\item Linear algebra and matrix theory: PC matrices are defined by algebraic relations between their elements.

\item Differential geometry: The subspaces generated in Grassmannians represent discrete tangent planes, providing a connection to geometric structures.

\item Algebraic geometry: Plücker coordinates, fundamental to algebraic geometry, ensure that configurations are valid under algebraic constraints.
\end{itemize}

\noindent The connection between PC matrices , Grassmannians, and Plücker coordinates naturally extends to the realm of differential geometry, particularly the theory of 2-forms. This section explores how skew-symmetric additive PC matrices can be interpreted as differential 2-forms, offering new insights and applications in mathematical modeling and theoretical physics.

\subsection*{Representation of PC matrices as differential 2-forms}

Let \( A = [a_{ij}] \) be a skew-symmetric additive PC matrix of order \( n \), where \( a_{ij} = s_i - s_j \) and \( s_i = \|v_i\|^2 \), with \( \{v_i\}_{i=1}^n \subseteq \mathbb{R}^n \) being a set of linearly independent vectors. We define a differential 2-form \( \omega \) on a subspace \( U \subseteq \mathbb{R}^n \) as:  
\[
\omega = \sum_{1 \leq k < l \leq n} p_{kl} \, dx_k \wedge dx_l,
\]
\noindent where \( p_{kl} = v_{i,k}v_{j,l} - v_{i,l}v_{j,k} \) are the Plücker coordinates of the 2-plane spanned by \( v_i \) and \( v_j \). 
\\
\\
The following properties hold: 
\begin{itemize}

\item Pairwise comparisons encoding: 
The PCM entries \( a_{ij} \) are directly related to \( \omega \) by the evaluation:
   \[
   \omega(v_i, v_j) = a_{ij}.
   \]
   
\item Consistency and closedness: 
The algebraic consistency condition \( a_{ij} + a_{jk} - a_{ki} = 0 \) corresponds to the geometric condition \( d\omega = 0 \), ensuring that \( \omega \) is closed. This closedness is guaranteed by the Plücker relations, which enforce the validity of \( \omega \) as a 2-form in the Grassmannian \( G(2, n) \).

\end{itemize}

\subsection*{Implications and applications}

This interpretation bridges the algebraic structure of PC matrices with differential geometry, enabling novel applications: 

\begin{itemize}

\item Geometric modeling: 
The consistency of a PC matrix ensures that the corresponding 2-form \( \omega \) represents a coherent geometric configuration, useful in optimization and design problems. 

\item Theoretical physics: 
In fields such as gauge theory and symplectic geometry, \( \omega \) can model physical systems where pairwise interactions or relationships are critical. 

\item Higher-dimensional generalization: 
The framework can extend to higher Grassmannians \( G(k, n) \), facilitating the study of multidimensional pairwise comparisonss.

\end{itemize}

\noindent This novel perspective opens new avenues for applying PCMs beyond their traditional scope, fostering interdisciplinary connections between linear algebra, geometry, and physics.

\section{Conclusion} 
This paper proposes a novel geometric framework for additive skew-symmetric pairwise comparisons matrices, embedding them into the Grassmannian \( G(2, n) \). The central result establishes a direct equivalence between the algebraic consistency condition \( a_{ij} + a_{jk} - a_{ki} = 0 \) and the geometric consistency of the corresponding subspaces in \( G(2, n) \), expressed through the satisfaction of the Plücker relations. The framework presented in this work is not intended to replace existing methods for handling pairwise comparison matrices, but rather represents a theoretical bridge between distinct mathematical branches. By introducing PC matrices into Grassmannians \( G(2, n) \) and using Plücker coordinates, this study opens new interdisciplinary perspectives to better understand their geometric and algebraic properties. This theoretical reinterpretation could form the basis for future applications in complex fields such as mathematical physics, optimization and differential geometry.
\\
In \textit{Section 5}, we explored how this geometric interpretation extends beyond Grassmannians to differential geometry, particularly the theory of 2-forms. The skew-symmetric nature of pairwise comparisons matrices allows us to represent them as closed differential 2-forms, where the consistency condition ensures their geometric coherence. This perspective links the algebraic structure of pairwise comparisons matrices with geometric and topological properties, suggesting potential applications in fields like gauge theory and symplectic geometry. 
\\
This approach not only extends the applicability of Grassmannians but also opens new avenues for exploring their role in understanding geometric consistency, algebraic structures, and their potential applications in theoretical physics and optimization.


\begin{thebibliography}{99}

\bibitem{Balinsky2024}
Balinsky, Alexander A., Bovdi, Victor A., Prykarpatski, Anatolij K., "On the Quantum Deformations of Associative Sato Grassmannian Algebras and the Related Matrix Problems", SYMMETRY, 16(1), 2024.

\bibitem{Calixto2022}
Calixto, M., Mayorgas, A., Guerrero, J., "Hilbert Space Structure of the Low Energy Sector of U(N) Quantum Hall Ferromagnets and Their Classical Limit", SYMMETRY, 14(5), 2022.

\bibitem{Fioresi2021}
Fioresi, R., Lledo, Maria A., "Quantum Supertwistors", SYMMETRY, 13(7), 2021.


\bibitem{Govc} 
Govc, D., Marzantowicz, W., and Pavesic, P., "How many simplices are needed to triangulate a Grassmannian?", TOPOLOGICAL METHODS IN NONLINEAR ANALYSIS,  56 (2) , pp.501-518, 2020.

\bibitem{Harris}
Harris, J., "Algebraic Geometry: A First Course", Springer, Graduate Texts in Mathematics, Vol. 133, 1992.

\bibitem{Jung2019}
Jung, S. M., Kim, Y. H., Qian, J., "New Characterizations of the Clifford Torus and the Great Sphere", SYMMETRY, 11(9), 2019.

\bibitem{K1993}
Koczkodaj, W.W., "A new definition of consistency of pairwise comparisons", MATHEMATICAL AND COMPUTER MODELLING, 18(7), 79–84, 1993. \url{https://doi.org/10.1016/0895-7177(93)90059-8}.

\bibitem{K1996}
Koczkodaj, WW., "Statistically accurate evidence of improved error rate by pairwise comparisons", PERCEPTUAL AND MOTOR SKILLS, 82(1): 43-48, 1996. DOI10.2466/pms.1996.82.1.43

\bibitem{KS2016}
Koczkodaj, W.W.; Szybowski, J., "The Limit of Inconsistency Reduction in Pairwise Comparisons", INTERNATIONAL JOURNAL OF APPLIED MATHEMATICS AND COMPUTER SCIENCE, 26(3): 721-729, 2016. \url{https://doi.org/10.1016/j.ijar.2022.10.005}

\bibitem{KSW2016}
Koczkodaj, W.W.; Szybowski, J.; Wajch, E., "Inconsistency indicator maps on groups for pairwise comparisons", INTERNATIONAL JOURNAL OF APPROXIMATE REASONING, 69: 81-90, 2016.
\url{https://doi.org/10.1016/j.ijar.2015.11.007}.

\bibitem{Kuo2021}
Kuo, T., "An ordinal consistency indicator for pairwise comparison matrix", SYMMETRY, 13(11), 2021.
DOI: 10.3390/sym13112183

\bibitem{Liu} 
Liu, W., Pankov, M., and Wang, KS., "Transformations of polar Grassmannians preserving certain intersecting relations", JOURNAL OF ALGEBRAIC COMBINATORICS, 40 (2) , pp.633-646, 2014.

\bibitem{SIR2019}
Mazurek, M., Strzalka, D., Wolny-Dominiak, A., Woodbury-Smith, M., "Electronic Health Record Breaches as Social Indicators", SOCIAL INDICATORS RESEARCH, 141(2): 861-871, 2019.
\url{https://doi.org/10.1007/s11205-018-1837-z}.

\bibitem{Pankov2006}
Pankov, M., Prazmowski, K., and ynel, M., "Geometry of polar Grassmann spaces", DEMONSTR. MATH., 39 , pp.625-637, 2006.

\bibitem{Pankov2014}
Pankov, M., "Characterization of isometric embeddings of Grassmann graphs", ADVANCES IN GEOMETRY, 14 (1) , pp.91-108, 2014.

\bibitem{Pankov2018}
Pankov, M., "Geometric version of Wigner's theorem for Hilbert Grassmannians", JOURNAL OF MATHEMATICAL ANALYSIS AND APPLICATIONS, 459 (1) , pp.135-144, 2018.

\bibitem{Planat2020}
Planat, M., Aschheim, R., Amaral, Marcelo M., Irwin, K., "Quantum Computation and Measurements from an Exotic Space-Time \(R^4\)", SYMMETRY, 12(5), 2020.

\bibitem{Prazmowska2008}
Prazmowska, M., Prazmowski, K., and Zynel, M., "Euclidean geometry of orthogonality of subspaces", AEQUATIONES MATHEMATICAE
 76 (1-2), pp.151-167, 2008.

\bibitem{Prazmowska2010}
Prazmowska, M., Prazmowski, K., and Zynel, M., "Grassmann spaces of regular subspaces", JOURNAL OF GEOMETRY, 97 (1-2) , pp.99-123, 2010.

\bibitem{Prazmowski1989}
Prazmowski, K., "An axiomatic description of strambach planes", GEOMETRIAE DEDICATA, 32 (2) , pp.125-156, 1989.

\bibitem{Tu}
Tu, L. W., "An Introduction to Manifolds", Springer, 2nd edition, 2010.

\bibitem{Yamawaki2023}
Yamawaki, K., "Proving Rho Meson Is a Dynamical Gauge Boson of Hidden Local Symmetry", SYMMETRY, 15(12), 2023

\bibitem{Zynel2014}
Zynel, M., "Complements of Grassmann substructures in projective Grassmannians", AEQUATIONES MATHEMATICAE, 88 (1-2) , pp.81-96, 2014.




\end{thebibliography}
\end{document}